\documentclass[12pt,a4paper,twoside]{amsart}
\usepackage[english]{babel} 
\usepackage[utf8]{inputenc}
\usepackage{amsmath,amssymb,amsthm,amscd}
\usepackage{graphicx}
\usepackage{fancyhdr}
\usepackage{enumerate}
\usepackage{indentfirst}
\usepackage{geometry}
\usepackage{hyperref}
\usepackage[all]{xy}
\usepackage{calc,color}
\newcommand{\R}{\ensuremath{\mathbb{R}}}
\newcommand{\N}{\ensuremath{\mathbb{N}}} 
\newcommand{\LL}{\ensuremath{\mathcal{L}}}
\newcommand{\delim}[3]{\left#1 #3 \right#2}  
\newcommand{\pin}[1]{\delim{\langle}{\rangle}{#1}} 
  
\def \tn {\textnormal}  
  
\theoremstyle{definition} 
\newtheorem{theo}{Theorem}[section] 
\newtheorem{lem}[theo]{Lemma} 
\newtheorem{cor}[theo]{Corollary} 
\newtheorem{prop}[theo]{Proposition} 
\newtheorem{defi}[theo]{Definition} 
\title[Rate of convergence of attractors]{Rate of convergence of attractors for semilinear singularly perturbed problems: parabolic equations with large diffusion}
\author[L. Pires]{Leonardo Pires$^*$}
\address[Leonardo Pires]{Departamento de Matem\'atica,
Instituto de Ci{\^e}ncias Matem{\'a}ticas e de Computa\c{c}{\~a}o,
Universidade de S{\~a}o Paulo--Campus de S{\~a}o Carlos, Caixa
Postal 668, 13560-970 S{\~a}o Carlos SP, Brazil}
\email{leopires@icmc.usp.br}

\begin{document}

\begin{abstract}
We exhibit a singularly perturbed parabolic problems for which the asymptotic behavior can be described by an one-dimensional ordinary differential equation.  We estimate the continuity of attractors in the Hausdorff metric by rate of convergence of resolvent operator.  
\end{abstract}

\maketitle

\section{Introduction, Functional Setting and Statement of the Results}
We consider the parabolic problem
\begin{equation}\label{perturbed_problem}
\begin{cases}
u^\varepsilon_t-\tn{div}(p_\varepsilon(x) \nabla u^\varepsilon)+(\lambda+V_\varepsilon(x))u^\varepsilon = f(u^\varepsilon),\quad x\in \Omega,\,\,t>0, \\
\dfrac{\partial u^\varepsilon}{\partial \vec{n}}=0,\quad x\in \partial \Omega,\\
u^\varepsilon(0)=u^\varepsilon_0,
\end{cases}
\end{equation}
where $0<\varepsilon<\varepsilon_0$, $\Omega\subset \R^n$ is a bounded smooth open connected set, $\partial \Omega$ is the boundary of $\Omega$ and $\frac{\partial u^\varepsilon}{\partial \vec{n}}$ is the co-normal derivative operator with $\vec{n}$ the unit outward normal vector to $\partial \Omega$. We assume the potentials $V_\varepsilon\in L^p(\Omega)$ with  
\begin{equation}
p\begin{cases} \geq 1,\quad n=1,\\\geq 2,\quad n\geq2,
\end{cases}
\end{equation}
and $V_\varepsilon$ converges to $V_0 \in \R$ in $L^p(\Omega)$, that is, we consider $\tau(\varepsilon)$ an increasing positive function of $\varepsilon$ such that
\begin{equation}\label{map_tau}
\|V_\varepsilon-V_0\|_{L^p(\Omega)}\leq \tau(\varepsilon)\overset{\varepsilon\to 0}\longrightarrow 0.
\end{equation}
Note that \eqref{map_tau} implies that the spatial average of $V_\varepsilon$ converges to $V_0$ as $\varepsilon\to 0$.  We choice $\lambda\in \R$ sufficiently large for that $\tn{ess}\inf_{x\in\Omega}V_\varepsilon(x)+\lambda\geq m_0$ for some positive constant $m_0$. Moreover we will assume the diffusion is large in $\Omega$, that is, for each $\varepsilon\in (0,\varepsilon_0)$ the map $p_\varepsilon$ is positive smooth defined in $\bar{\Omega}$ satisfying
$$
p(\varepsilon):=\min_{x\in \bar{\Omega}} \{p_\varepsilon(x)\}\overset{\varepsilon\to 0}\longrightarrow \infty\quad\tn{with}\quad 0<m_0\leq p_\varepsilon(x),\quad\forall\,x\in\Omega.
$$ 
Since large diffusivity implies fast homogenization, we expect, for small values of $\varepsilon$, that the solution of this problem converge to a constant spatial function in $\Omega$. Indeed by taking the average on $\Omega$, the limiting problem as $\varepsilon$ goes to zero is given by ordinary differential equation
\begin{equation}\label{limit_problem}
\begin{cases}
\dot{u}^0+(\lambda+V_0)u^0=f(u^0),\quad t>0,\\
u^0(0)=u^0_0,
\end{cases}
\end{equation}
which \cite{Rodriguez-Bernal2005} proves to determine the asymptotic behavior.

In this paper we are concerning in how fast the dynamics of the problem \eqref{perturbed_problem} approaches the dynamics of the problem \eqref{limit_problem}. We will estimate this convergence by functions $\tau(\varepsilon)$ and $p(\varepsilon)$.

Since we have established the limit problem we need to study the well posedness of \eqref{perturbed_problem} and \eqref{limit_problem} as abstract parabolic equation in appropriated Banach spaces. To that end, we define the operator $A_\varepsilon:\mathcal{D}(A_\varepsilon)\subset L^2(\Omega)\to L^2(\Omega)$ by
$$
\mathcal{D}(A_\varepsilon)=\{u\in H^2(\Omega)\,;\, \dfrac{\partial u^\varepsilon}{\partial \vec{n}}=0\},\quad
A_\varepsilon u=-\tn{div}(p_\varepsilon \nabla u)+(\lambda+V_\varepsilon )u.
$$
We denote 
$
L^2_{\Omega}=\{u\in H^1(\Omega)\,;\, \nabla u=0 \tn{ in } \Omega\}
$
and we define the operator $A_0:L^2_{\Omega}\subset L^2(\Omega)\to L^2(\Omega)$ by
$
A_0 u=(\lambda+V_0)u.
$

It is well known that $A_\varepsilon$ is a positive invertible operator with compact resolvent for each $\varepsilon\in [0,\varepsilon_0]$, hence we define in the usual way (see\cite{Henry1980}), the fractional power space $X_\varepsilon^\frac{1}{2}=H^1(\Omega)$, $\varepsilon\in (0,\varepsilon_0]$, and $X_0^\frac{1}{2}=L^2_{\Omega}$ with the scalar products 
$$
\pin{u,v}_{X_\varepsilon^{\frac{1}{2}}}=\int_\Omega p_\varepsilon \nabla u \nabla v\,dx+\int_\Omega(\lambda +V_\varepsilon)uv\,dx,\quad u,v\in X_\varepsilon^\frac{1}{2},\,\,\,\varepsilon\in (0,\varepsilon_0];
$$
$$
\pin{u,v}_{X_0^{\frac{1}{2}}}=|\Omega|^{-1}(\lambda+V_0)uv,\quad u,v\in X_0^\frac{1}{2}.
$$
The space $X_0^\frac{1}{2}$ is a one dimensional closed subspace of $X_\varepsilon^\frac{1}{2}$, $\varepsilon\in (0,\varepsilon_0]$ and $X_\varepsilon^\frac{1}{2}\subset H^1(\Omega)$ with injection constant independent of $\varepsilon$, but the injection $H^1(\Omega)\subset X_\varepsilon^\frac{1}{2}$ is not uniform, in fact is valid
\begin{equation}\label{non_equivalence_norm}
m_0\|u\|_{H^1}^2\leq \|u\|_{X_\varepsilon^\frac{1}{2}}^2\leq M(\varepsilon) \|u\|^2_{H^1}, 
\end{equation}
with $M(\varepsilon)\to \infty$ as $\varepsilon\to 0$ and we will show in the Corollary \ref{non_equivalence_norm1} that there is no positive constant $C$ independent of $\varepsilon$ such that 
$
\|u\|_{X_\varepsilon^\frac{1}{2}}^2\leq C\|u\|^2_{H^1}.
$
Therefore bounds for solutions in the Sobolev spaces does not give suitable estimates in the fractional power space, even though we will consider $X_\varepsilon^\frac{1}{2}$ as phase space. 

If we denote the Nemitskii functional of $f$ by the same notation $f$, then \eqref{perturbed_problem} and \eqref{limit_problem} can be written as  
\begin{equation}\label{semilinear_problem}
\begin{cases}
u^\varepsilon_t+A_\varepsilon u^\varepsilon = f(u^\varepsilon), \\
u^\varepsilon(0)=u_0^\varepsilon \in X_\varepsilon^{\frac{1}{2}},\quad \varepsilon\in [0,\varepsilon_0].
\end{cases}
\end{equation}  

We assume $f$ is continuously differentiable and the equilibrium set of \eqref{semilinear_problem} for $\varepsilon=0$ is composed of a finite number of hyperbolic equilibrium points. That is 
$$
\mathcal{E}_0:=\{u\in D(A_0)\,;\,A_0u-f(u)=0\}=\{x_*^{1,0}<x_*^{2,0}\leq ...\leq x_*^{m,0}\}
$$
and $\sigma(A_0-f'(u_*^{i,0}))\cap\{\mu\,;\,Re(\mu)=0\}=\emptyset$, for $i\in\{1,...,m\}$.

In order to ensure that all solution of \eqref{semilinear_problem} are globally defined, and there is a global attractor for the nonlinear semigroup given by theses solutions, we assume the following conditions. 
\begin{itemize}
\item[(i)]  If $n=2$, for all $\eta>0$, there is a constant $C_\eta>0$ such that
$$
 |f(u)-f(v)|\leq C_\eta (e^{\eta|u|^2}+e^{\eta|v|^2})|u-v|,\quad \forall\,u,v\in \R,
$$     
and if $n\geq 3$, there is a constant $\tilde{C}>0$ such that
$$
|f(u)-f(v)|\leq \tilde{C}|u-v| (|u|^{\frac{4}{n-2}}+|v|^{\frac{4}{n-2}}+1 ),\quad \forall\,u,v\in \R.
$$     
\item[(ii)] 
$$
\limsup_{|u|\to\infty} \dfrac{f(u)}{u} <0.
$$
\end{itemize}
Under theses assumptions \cite{J.M.Arrieta1999,J.M.Arrieta2000} and \cite{Hale1988} ensure that the problem \eqref{semilinear_problem} is globally well posed and generate a nonlinear semigroup satisfying
\begin{equation}\label{nonlinear_semigroup}
T_\varepsilon(t)u_0^\varepsilon=e^{-At}u_0^\varepsilon+\int_0^t e^{-A(t-s)}f(T_\varepsilon(s)u_0^\varepsilon)\,ds,\quad t\geq 0.
\end{equation}
Moreover there is a global attractor $\mathcal{A}_\varepsilon$ for $T_\varepsilon(\cdot)$ uniformly bounded in $X_\varepsilon^\frac{1}{2}$, that is
$$
\sup_{\varepsilon\in [0,\varepsilon_0]}\sup_{w\in \mathcal{A}_\varepsilon}\|w\|_{X_\varepsilon^\frac{1}{2}}<\infty.
$$ 
We also have $T_0(\cdot)$ is a Morse-Smale semigroup (see \cite{Bortolan}) and $\mathcal{A}_0=[x_*^{1,0},x_*^{m,0}]$.  

We recall that 
\begin{defi}[Invariant Manifold]\label{invariant_manifold}
A set $\mathcal{M}_\varepsilon\subset X_\varepsilon^\frac{1}{2}$ is an invariant manifold for \eqref{semilinear_problem} when for each $u_0^\varepsilon\in \mathcal{M}_\varepsilon$ there is a global solution $u^\varepsilon(\cdot)$ of \eqref{semilinear_problem} such that $u^\varepsilon(0)=u_0^\varepsilon$ and $u^\varepsilon(t)\in\mathcal{M}_\varepsilon$ for all $t\in\R$.
\end{defi}

For $z^\varepsilon\in X_\varepsilon^\frac{1}{2}$ and $A,B\subset X_\varepsilon^\frac{1}{2},$ we denote
$$
\tn{dist}(z^\varepsilon,B)=\inf_{x\in B}\|z-x\|_{X_\varepsilon^\frac{1}{2}}\quad\tn{and}\quad
\tn{dist}_H(A,B)=\sup_{z\in A}\tn{dist}(z,A)
$$
The Hausdorff metric is defined by
$$
\tn{d}_H(A,B)=\tn{dist}_H(A,B)+\tn{dist}_H(B,A).
$$

We consider the projection 
$$
Pu=\frac{1}{|\Omega|}\int_\Omega u \,dx,\quad u\in L^2(\Omega)\quad\tn{or}\quad u\in X_\varepsilon^\frac{1}{2}.
$$
Thus $P$ is an orthogonal projection acting on $L^2$ onto $L^2_\Omega$ or $X_\varepsilon^\frac{1}{2}$ onto $X_0^\frac{1}{2}$.
 
We are now in position to state the main result. 
\begin{theo}\label{main_result} For $\varepsilon\in (0,\varepsilon_0]$ there is a positive constant $C$ independent of $\varepsilon$ such that the operators $A_\varepsilon$ satisfy
$$
\|A_\varepsilon^{-1}-A_0^{-1}P\|_{\LL(L^2,X_\varepsilon^\frac{1}{2})}\leq C(\tau(\varepsilon)+p(\varepsilon)^{-\frac{1}{2}}).
$$
There is an one dimensional invariant manifold $\mathcal{M}_\varepsilon$ for \eqref{semilinear_problem} such that $\mathcal{A}_\varepsilon\subset \mathcal{M}_\varepsilon$ and the flow on $\mathcal{A}_\varepsilon$ can be reduced to an ordinary differential equation. Moreover the continuity of attractors of \eqref{nonlinear_semigroup} can be estimate by   
$$
\tn{d}_H(\mathcal{A}_\varepsilon,\mathcal{A}_0)\leq C(\tau(\varepsilon)+p(\varepsilon)^{-\frac{1}{2}}).
$$
\end{theo}
The proof of Theorem \ref{main_result} will be done in the Corollary \ref{rate_of_resolvent} and Theorem \ref{rate_attractors} in the following sections.

The paper is divided as follow: in Section \ref{Rate of Convergence of Resolvent and Equilibria} we make a detailed study of operators $A_\varepsilon$, $\varepsilon\in [0,\varepsilon_0]$ obtaining the rate of convergence of resolvent operators and equilibrium points. In section \ref{Invariant_Manifold} we construct the invariant manifold and in the Section \ref{Rate_of_Convergence_of_Attractors} we reduce the system to one dimensional in order to obtain the rate of convergence of attractors.

\section{Rate of Convergence of Resolvent and Equilibria}\label{Rate of Convergence of Resolvent and Equilibria}

In this section we will obtain the convergence of the resolvent operators and the convergence of equilibrium points. We will study the spectral behavior of the operators $A_\varepsilon$ and we will obtain estimates for the linear semigroups.

\begin{lem}\label{Rate_of_convergence}
For $g\in L^2(\Omega)$ with $\|g\|_{L^2}\leq 1$ and $\varepsilon\in (0,\varepsilon_0]$, let $u^\varepsilon$ be the solution of elliptic problem  
\begin{equation}
\begin{cases}
-\tn{div}(p_\varepsilon(x)\nabla u^\varepsilon)+(\lambda+V_\varepsilon(x))u^\varepsilon=g, \quad x\in \Omega,\\
\dfrac{\partial u^\varepsilon}{\partial \vec{n}}=0,\quad x\in \partial\Omega.
\end{cases}
\end{equation}
Then there is a constant $C>0$, independent of $\varepsilon$, such that
\begin{equation}\label{convergence_resolvent0}
\|u^\varepsilon-u^0\|_{X_\varepsilon^\frac{1}{2}}\leq C(\tau(\varepsilon)+p(\varepsilon)^{-\frac{1}{2}}),
\end{equation} 
where $u^0=\frac{Pg}{\lambda+V_0}$.
\end{lem}
\begin{proof}
The weak solution $u^\varepsilon$ satisfies
\begin{equation}\label{weak_solution}
\int_\Omega p_\varepsilon \nabla u^\varepsilon \nabla \varphi\,dx+\int_\Omega (\lambda+V_\varepsilon)u^\varepsilon \varphi\,dx=\int_\Omega g\varphi\,dx,\quad\forall\,\varphi\in X_\varepsilon^{\frac{1}{2}},\,\,\varepsilon\in(0,\varepsilon_0];
\end{equation}
\begin{equation}\label{weak_solution2}
\int_\Omega (\lambda+ V_0)u^0 \varphi\,dx=\int_\Omega Pg \varphi\,dx,\quad\forall\,\varphi\in X_0^{\frac{1}{2}}.
\end{equation}
With the aid of the projection $P$ we have
$$
\int_\Omega p_\varepsilon |\nabla u^\varepsilon|^2\,dx+\int_\Omega (\lambda+V_\varepsilon)u^\varepsilon(u^\varepsilon-u^0)\,dx=\int_\Omega g(u^\varepsilon-u^0)\,dx;
$$
$$
\int_\Omega (\lambda+V_0)u^0(Pu^\varepsilon-u^0)\,dx=\int_\Omega Pg(Pu^\varepsilon-u^0)\,dx,
$$
which implies
$$
\int_\Omega g(u^\varepsilon-u^0)\,dx-\int_\Omega Pg(Pu^\varepsilon-u^0)\,dx=\int_\Omega g(I-P)u^\varepsilon\,dx
$$
and
\begin{align*}
\int_\Omega p_\varepsilon |\nabla u^\varepsilon|^2\,dx &+\int_\Omega (\lambda+V_\varepsilon )u^\varepsilon(u^\varepsilon-u^0)\,dx-\int_\Omega (\lambda+V_0) u^0(Pu^\varepsilon-u^0)\,dx\\ 
&=\|u^\varepsilon-u^0\|_{X_\varepsilon^\frac{1}{2}}^2+\int_\Omega (V_\varepsilon-V_0)u^0(u^\varepsilon-u^0)\,dx.
\end{align*}
Therefore 
$$
\|u^\varepsilon-u^0\|_{X_\varepsilon^\frac{1}{2}}^2\leq \int_\Omega |V_\varepsilon-V_0||u^0||u^\varepsilon-u^0|\,dx+\int_\Omega |g(I-P)u^\varepsilon|\,dx.
$$
If $n=1$, we have $X_\varepsilon^\frac{1}{2}\subset H^1\subset L^\infty$, thus
$$
\int_\Omega |V_\varepsilon-V_0||u^0||u^\varepsilon-u^0|\,dx\leq C\|u^\varepsilon-u^0\|_{L^\infty}\|V^\varepsilon-V_0\|_{L^1}\leq C\|u^\varepsilon-u^0\|_{X_\varepsilon^\frac{1}{2}}\tau(\varepsilon).
$$
If $n\geq 2$, we have $L^p\subset L^2$, thus
$$
\int_\Omega |V_\varepsilon-V_0||u^0||u^\varepsilon-u^0|\,dx\leq C\|u^\varepsilon-u^0\|_{L^2}\|V^\varepsilon-V_0\|_{L^2}\leq C\|u^\varepsilon-u^0\|_{X_\varepsilon^\frac{1}{2}}\tau(\varepsilon).
$$
By Poincaré's inequality for average, we have
$$
\int_\Omega |g(I-P)u^\varepsilon|\,dx\leq \|g\|_{L^2}\Big(\int_\Omega |\nabla u^\varepsilon|^2\,dx\Big)^\frac{1}{2},
$$
but
$$
p(\varepsilon)\int_\Omega|\nabla u^\varepsilon|^2\,dx \leq\int_\Omega p_\varepsilon|\nabla u^\varepsilon-\nabla u^0|^2\,dx\leq \|u^\varepsilon-u^0\|_{X_\varepsilon^\frac{1}{2}}^2.
$$
Put this estimates together the result follows. 
\end{proof}

\begin{cor}\label{non_equivalence_norm1}
There is no positive constant $C$ independent of $\varepsilon$ such that 
$$
\|u\|_{X_\varepsilon^\frac{1}{2}}^2\leq C\|u\|_{H^1}^2\quad \forall\,u\in X_\varepsilon^\frac{1}{2}.
$$
\end{cor}
\begin{proof}
If there is such a constant $C$ take $v^\varepsilon=(I-P)u^\varepsilon$ as given by the Lemma \ref{rate_of_resolvent}, thus by Poincaré's inequality for average, we have
\begin{align*}
p(\varepsilon)\|v^\varepsilon\|_{H^1}^2 & \leq p(\varepsilon)\int_\Omega |\nabla v^\varepsilon|^2\,dx+\int_\Omega |v^\varepsilon|^2\,dx\\
& \leq Cp(\varepsilon)\int_\Omega |\nabla v^\varepsilon|^2\,dx\\
&\leq C\|v^\varepsilon\|_{X_\varepsilon^\frac{1}{2}}^2\leq C\|v^\varepsilon\|_{H^1}^2.
\end{align*}
\end{proof}

The convergence of the resolvent operators can be stated as follows.  
\begin{cor}\label{rate_of_resolvent}
There is a positive constant $C$ independent of $\varepsilon$ such that
\begin{equation}
\|A_\varepsilon^{-1}-A_0^{-1}P\|_{\LL(L^2),X_\varepsilon^\frac{1}{2})}\leq C(\tau(\varepsilon)+p(\varepsilon)^{-\frac{1}{2}}).
\end{equation} 
Furthermore, if $\mu \in \rho(-A_0)\cap \rho(-A_\varepsilon)$ is such that Re$(\mu)\notin (-\infty,\bar{\lambda}]$, where $\bar{\lambda}= \lambda+V_0$, there is $\phi\in (\frac{\pi}{2},\pi)$ such that for all $\mu \in\Sigma_{\bar{\lambda},\phi}=\{\mu\in\mathbb{C}\,:\,|\tn{arg}(\mu+\bar{\lambda})|\leq \phi \}\setminus \{\mu\in\mathbb{C}:|\mu+\bar{\lambda}|\leq r\}$, for some $r>0$,  
\begin{equation}
\|(\mu+A_\varepsilon)^{-1}-(\mu+A_0)^{-1}P\|_{\LL(L^2,X_\varepsilon^{\frac{1}{2}})}\leq C(\tau(\varepsilon)+p(\varepsilon)^{-\frac{1}{2}}).
\end{equation}
\end{cor}

The spectral behavior of $A_\varepsilon$ can be seen in the next result. 
\begin{prop}\label{spectral_properties} We denote $\bar{\lambda}=\lambda+V_0$ the eigenvalue of $A_0$ and we denote for $\varepsilon\in (0,\varepsilon_0]$ the ordered spectrum  $\sigma(A_\varepsilon)=\{\lambda_1^\varepsilon<\lambda_2^\varepsilon<...\}$.
\begin{itemize}
\item[(i)] Given $\delta>0$, there is $\varepsilon$ sufficiently small (we still denote $\varepsilon\in (0,\varepsilon_0]$) such that the operators
$$
Q_\varepsilon(\bar{\lambda})=\frac{1}{2\pi i}\int_{|\xi+\bar{\lambda}|=\delta}(\xi+A_\varepsilon)^{-1}\,d\xi,\quad \varepsilon\in[0,\varepsilon_0],
$$ 
are projections on $X_\varepsilon^\frac{1}{2}$ and $Q_\varepsilon(\bar{\lambda})\longrightarrow Q_0(\bar{\lambda})=I_{X_0^\frac{1}{2}}$. More precisely,
\begin{equation}\label{projection_convergence}
\|Q_\varepsilon(\bar{\lambda})-P\|_{\LL(L^2,X_\varepsilon^\frac{1}{2})}\leq C(\tau(\varepsilon)+p(\varepsilon)^{-\frac{1}{2}}),
\end{equation}
for some constant $C$ independent of $\varepsilon$. Furthermore all  eigenspaces $W_\varepsilon(\bar{\lambda})=Q_\varepsilon(\bar{\lambda})X_\varepsilon^\frac{1}{2}$ are the same dimension, that is, for $\varepsilon$ sufficiently small,
$$
\tn{rank}(Q_\varepsilon(\bar{\lambda}))=\tn{dim}(W_\varepsilon(\bar{\lambda}))=\tn{dim}(X_0^\frac{1}{2})=1.
$$
\item[(ii)]For each $u^0\in W_0(\bar{\lambda})=X_0^\frac{1}{2}$ there is a sequence $(u^{\varepsilon_k})_k$ such that $u^{\varepsilon_k}\in W_{\varepsilon_k}(\bar{\lambda})$, $\varepsilon_k\overset{k\to\infty}\longrightarrow 0$ and $u^{\varepsilon_k}\to u^0$ in $X_\varepsilon^\frac{1}{2}$. Also, given sequences $\varepsilon_k\overset{k\to\infty}\longrightarrow 0$ and $(u^{\varepsilon_k})_k$ with $u^{\varepsilon_k}\in W_{\varepsilon_k}(\bar{\lambda})$ and $\|u^{\varepsilon_k}\|_{X_\varepsilon^\frac{1}{2}}=1$, $k\in\N$, each subsequence $(u^{\varepsilon_{k_l}})_l$ of $(u^{\varepsilon_k})_k$ has a  convergent subsequence for some $u^0\in W_0(\bar{\lambda})$. Furthermore, if we denote $W_\varepsilon^1(\bar{\lambda})=\{z\in W_\varepsilon(\bar{\lambda})\,;\,\|z\|_{X_\varepsilon^\frac{1}{2}}\leq 1\}$, then
\begin{equation}\label{eigenspace_convergence}
\tn{d}_H(W^1_\varepsilon(\bar{\lambda}),W^1_0(\bar{\lambda}))\leq C (\tau(\varepsilon)+p(\varepsilon)^{-\frac{1}{2}}).
\end{equation}
\item[(iii)]Given $R>0$, for $\varepsilon$ sufficiently small (we still denote $\varepsilon\in (0,\varepsilon_0]$), such that, $|\lambda_1^\varepsilon-\bar{\lambda}|\leq R$ and $|\lambda_2^\varepsilon|>R$. Furthermore
\begin{equation}\label{eigenvalue_convergence}
C p(\varepsilon)\leq \lambda^j_\varepsilon,\quad j\in\{2,3,...\}.
\end{equation}
\end{itemize}
\end{prop}
\begin{proof}
The proof is the same as given in \cite{Carbone2008} and \cite{Carvalho2010a}. Here we just need to proof the estimates \eqref{projection_convergence}, \eqref{eigenspace_convergence} and \eqref{eigenvalue_convergence}. But \eqref{projection_convergence} is immediately from definition of $Q_\varepsilon(\bar{\lambda})$ and by Corollary \ref{rate_of_resolvent}. Indeed
\begin{align*}
\|Q_\varepsilon(\bar{\lambda})-P\|_{\LL(L^2,X_\varepsilon^{\frac{1}{2}})}&=\Big\|\dfrac{1}{2\pi i} \int_{|z-\bar{\lambda}|=\delta} (z+A_\varepsilon)^{-1}-(z+A_0)^{-1}P\, dz\Big\|_{\LL(L^2,X_\varepsilon^{\frac{1}{2}})} \\
& \leq \dfrac{1}{2\pi} \int_{|z-\bar{\lambda}|=\delta} \|(z+A_\varepsilon)^{-1}-(z+A_0)^{-1}P\|_{\LL(L^2,X_\varepsilon^{\frac{1}{2}})}\, |dz|\\
&\leq C(\tau(\varepsilon)+p(\varepsilon)^{-\frac{1}{2}}).
\end{align*}
For \eqref{eigenspace_convergence} let $z_0\in X_0^\frac{1}{2}$ with $\|z_0\|_{X_0^\frac{1}{2}}\leq 1$, then 
$$
\tn{dist}(Q_\varepsilon(\bar{\lambda})z_0,W_0^1(\bar{\lambda}))\leq \|Q_\varepsilon(\bar{\lambda})z_0-z_0\|_{X_\varepsilon^\frac{1}{2}}\leq C(\tau(\varepsilon)+p(\varepsilon)^{-\frac{1}{2}});
$$
$$
\tn{dist}(z_0,W_\varepsilon^1(\bar{\lambda}))\leq \|z_0-Q_\varepsilon(\bar{\lambda})z_0\|_{X_\varepsilon^\frac{1}{2}}\leq C(\tau(\varepsilon)+p(\varepsilon)^{-\frac{1}{2}}),
$$
and the result follows. Finally for \eqref{eigenvalue_convergence}, let $\varphi_1^\varepsilon$ be the associated eigenfunction to the first eigenvalue $\lambda_1^\varepsilon\in \sigma(A_\varepsilon) $. Let $w\in X_\varepsilon^\frac{1}{2}\setminus \{0\}$ such that $w\perp \varphi_1^\varepsilon$ in $L^2(\Omega)$. We have
$$
\frac{\pin{w,w}_{X_\varepsilon^\frac{1}{2}}}{\int_\Omega |w|^2\,dx}\geq \frac{p(\varepsilon)\int_\Omega|\nabla w|^2\,dx}{\int_\Omega |w|^2\,dx},
$$
minimizing and setting $\mu_{2,\varepsilon}=\inf\{\int_\Omega|\nabla w|^2\,dx\,;\,w\in H^1\setminus\{0\},\,\|w\|_{L^2}=1\,w\perp \varphi_1^\varepsilon\}$, by min-max characterization,
$$
\lim_{\varepsilon\to 0}\frac{\lambda_\varepsilon^2}{p(\varepsilon)}\geq \lim_{\varepsilon\to 0}\mu_{2,\varepsilon},
$$
but $\mu_{2,\varepsilon}\overset{\varepsilon\to 0}\longrightarrow \mu_2$, where $\mu_2$ is the first non zero eigenvalue of the operator $-\Delta$ subject to Neumann homogeneous boundary condition (see \cite{Rodriguez-Bernal2005}). 
\end{proof}

The next result exhibit estimates that plays an important role in the existence of invariant manifolds in the next section.
\begin{prop}\label{Proposition_linear_estimates}
Choose $\delta>0$ and $R>2+\bar{\lambda}$ according with (iv) of Proposition \ref{spectral_properties} and denote
$$
Y_\varepsilon=Q_\varepsilon(\bar{\lambda})X_\varepsilon^\frac{1}{2}\quad\tn{and}\quad Z_\varepsilon=(I-Q_\varepsilon(\bar{\lambda}))X_\varepsilon^\frac{1}{2},\quad \varepsilon\in [0,\varepsilon_0]
$$
and define projected operators
$$
A_\varepsilon^+=A_\varepsilon|_{Y_\varepsilon}\quad\tn{and}\quad A_\varepsilon^-=A_\varepsilon|_{Z_\varepsilon},\quad \varepsilon\in [0,\varepsilon_0].
$$
Then the following estimates are true,
\begin{itemize}
\item[(i)]$\|e^{-A_\varepsilon^+  t}z\|_{X_\varepsilon^\frac{1}{2}}\leq M e^{-\gamma t}\|z\|_{X_\varepsilon^\frac{1}{2}},\quad t\leq 0,\quad z\in Y_\varepsilon;$
\item[(ii)]$\|e^{-A_\varepsilon^+ t}-e^{A_0^+ t}P\|_{\LL(L^2,X_\varepsilon^\frac{1}{2})}\leq Me^{-\gamma t}(\tau(\varepsilon)+p(\varepsilon)^{-\frac{1}{2}}), \quad t\leq 0;$
\item[(iii)] $\|e^{-A^-_\varepsilon  t}z\|_{X_\varepsilon^\frac{1}{2}}\leq Me^{-\beta(\varepsilon) t}\|z\|_{X_\varepsilon^\frac{1}{2}} ,\quad t> 0,\quad z\in Z_\varepsilon,$
\end{itemize}
where $\gamma=\bar{\lambda}+1$, $\beta(\varepsilon)=\lambda_\varepsilon^2$ and $M$ is a constant independent of $\varepsilon$.
\end{prop}
\begin{proof}
We have $\bar{\lambda}>0$ and for $\varepsilon$ sufficiently small we can construct the curve $\Gamma=\Gamma_1+\Gamma_2+\Gamma_3+\Gamma_4$, where  
$$
\Gamma_1=\{\mu\in\mathbb{C}\,;\,\tn{Re}(\mu)=-\bar{\lambda}+1\tn{ and }|\tn{Im}(\mu)|\leq 1\},
$$
$$
\Gamma_2=\{\mu\in\mathbb{C}\,;\,-\bar{\lambda}-1 \leq\tn{Re}(\mu)\leq-\bar{\lambda}+1\tn{ and }\tn{Im}(\mu)=1\},
$$
$$
\Gamma_3=\{\mu\in\mathbb{C}\,;\,\tn{Re}(\mu)=-\bar{\lambda}-1\tn{ and }|\tn{Im}(\mu)|\leq 1\},
$$
$$
\Gamma_4=\{\mu\in\mathbb{C}\,;\,-\bar{\lambda}-1 \leq\tn{Re}(\mu)\leq-\bar{\lambda}+1\tn{ and }\tn{Im}(\mu)=-1\}.
$$
Thus, for $z\in Y_\varepsilon$ and $t<0$, we  have
\begin{align*}
\|e^{-A_\varepsilon^+  t}z\|_{X_\varepsilon^\frac{1}{2}} & = \|e^{-A_\varepsilon t}Q_\varepsilon(\bar{\lambda}) z\|_{X_\varepsilon^\frac{1}{2}}\leq\frac{1}{2\pi}\int_\Gamma \|(\mu+A_\varepsilon)^{-1}e^{\mu t}z\|_{X_\varepsilon^\frac{1}{2}}\,d\mu\\
& \leq M_1 \int_\Gamma |e^{\mu t}|\,|d\mu| \|z\|_{X_\varepsilon^\frac{1}{2}}\leq M_2 \sup_{\mu\in\Gamma} e^{\tn{Re}(\mu) t} \|z\|_{X_\varepsilon^\frac{1}{2}} \\
& \leq M e^{-(\bar{\lambda}+1)t} \|z\|_{X_\varepsilon^\frac{1}{2}},
\end{align*}
which proves (i). The same argument proves (ii) with the aid of Lemma \ref{Rate_of_convergence}. For (iii), due our choices, $\lambda_\varepsilon^2$ does not lying in the region delimited by $\Gamma$ and then   
$$
e^{-A_\varepsilon^- t}z=e^{-A_\varepsilon t}(I-Q_\varepsilon(\lambda))z=\sum_{j=2}^\infty e^{-\lambda_\varepsilon^j t}\pin{z,\varphi_\varepsilon^j}_{L^2}\varphi_\varepsilon^j,\quad t>0,
$$
where $\varphi_\varepsilon^j$ are the orthonormal eigenfunctions associated to $\lambda_\varepsilon^j$, $j\in\{2,3,...\}$.  Then,
$$
\|e^{-A^-_\varepsilon  t}z\|_{X_\varepsilon^\frac{1}{2}}\leq e^{-\lambda_\varepsilon^2 t}\Big(\sum_{j=2}^\infty \pin{z,\varphi_\varepsilon^j}^2_{L^2}\lambda_\varepsilon^j\Big)^\frac{1}{2}\leq M e^{-\lambda_\varepsilon^2 t}\|z\|_{X_\varepsilon^\frac{1}{2}},\quad t>0.
$$
\end{proof}

The rate of convergence of equilibrium points can be obtained as follows.
\begin{theo}
Let $u_*^0\in\mathcal{E}_0$. Then for $\varepsilon$ sufficiently small (we still denote $\varepsilon\in (0,\varepsilon_0]$), there is $\delta>0$ such that the equation $A_\varepsilon u-f(u)=0$ has the only solution $u_*^\varepsilon\in \{u\in X_\varepsilon^\frac{1}{2}\,;\,\|u-u_*^0\|_{X_\varepsilon^\frac{1}{2}}\leq \delta\}$. Moreover
\begin{equation}\label{rate_of_equilibrium}
\|u_*^\varepsilon-u_*^0\|_{X_\varepsilon^\frac{1}{2}}\leq C(\tau(\varepsilon)+p(\varepsilon)^{-\frac{1}{2}}).
\end{equation} 
\end{theo}
\begin{proof}
The proof is the same as given in \cite{Arrieta} and \cite{Carbone2008}. Here we just need to proof the estimates \eqref{rate_of_equilibrium}. We have $u_*^\varepsilon$ and $u_*^0$ given by 
$$
u_*^0=(A_0+V_0)^{-1}[f(u_*^0)+V_0u_*^0]\quad\tn{and}\quad u_*^\varepsilon=(A_\varepsilon+V_0)^{-1}[f(u_*^\varepsilon)+V_0 u_*^\varepsilon],
$$ 
where $V_0=-f'(u_*^0)$. Thus
\begin{align*}
\|u_*^\varepsilon-u_*^0\|_{X_\varepsilon^\frac{1}{2}}&\leq \|(A_\varepsilon+V_0)^{-1}[f(u_*^\varepsilon)+V_0 u_*^\varepsilon]-(A_0+V_0)^{-1}[f(u_*^0)+V_0u_*^0]\|_{X_\varepsilon^\frac{1}{2}}\\
&\leq \|[(A_\varepsilon+V_0)^{-1}-(A_0+V_0)^{-1}P][f(u_*^\varepsilon)+V_0 u_*^\varepsilon]\|_{X_\varepsilon^\frac{1}{2}}\\
&+\|(A_0+V_0)^{-1}P[f(u_*^\varepsilon)-f(u_*^0)+V_0(u_*^\varepsilon-u_*^0)]\|_{X_\varepsilon^\frac{1}{2}}.
\end{align*}
We have the following equality 
$$
(A_\varepsilon+V_0)^{-1}-(A_0+V_0)^{-1}P=[I-(A_\varepsilon+V_0)^{-1}V_0](A_\varepsilon^{-1}-A_0^{-1}P)[I-V_0(A_0+V_0)^{-1}].
$$
And then
$
\|[(A_\varepsilon+V_0)^{-1}-(A_0+V_0)^{-1}P][f(u_*^\varepsilon)+V_0 u_*^\varepsilon]\|_{X_\varepsilon^\frac{1}{2}}\leq C(\tau(\varepsilon)+p(\varepsilon)^{-\frac{1}{2}}).
$

If we denote $z^\varepsilon=f(u_*^\varepsilon)-f(u_*^0)+V_0(u_*^\varepsilon-u_*^0)$, since $f$ is continuously differentiable, for all $\delta>0$ there is $\varepsilon$ sufficiently small such that 
$
\|z^\varepsilon\|_{X_\varepsilon^\frac{1}{2}}\leq \delta\|u_*^\varepsilon-u_*^0\|_{X_\varepsilon^\frac{1}{2}},
$
thus
$$
\|(A_0+V_0)^{-1}Pz^\varepsilon\|_{X_\varepsilon^\frac{1}{2}}\leq \delta \|(A_0+V_0)^{-1}P\|_{\LL(L^2,X_\varepsilon^\frac{1}{2})}\|u_*^\varepsilon-u_*^0\|_{X_\varepsilon^\frac{1}{2}}. 
$$
We choice $\delta$ sufficiently small such that $\delta\|(A_0+V_0)^{-1}P\|_{\LL(L^2,X_\varepsilon^\frac{1}{2})}\leq\frac{1}{2}$, and then
$$
\|u_*^\varepsilon-u_*^0\|_{X_\varepsilon^\frac{1}{2}}\leq C(\tau(\varepsilon)+p(\varepsilon)^{-\frac{1}{2}})+\frac{1}{2}\|u_*^\varepsilon-u_*^0\|_{X_\varepsilon^\frac{1}{2}}.
$$
\end{proof}

\section{Invariant Manifold}\label{Invariant_Manifold}
In this section we construct exponentially attracting invariant manifolds for \eqref{semilinear_problem} as the graph of Lipschitz continuous map defined an a suitable one dimensional space. We follow the original idea of \cite{Henry1980} carried out by \cite{Carvalho2010a}. 

\begin{theo}
Choose $\delta>0$ and $R>2+\bar{\lambda}$ according with Proposition \ref{spectral_properties} and for $\varepsilon\in [0,\varepsilon_0]$, write $X_\varepsilon^\frac{1}{2}=Y_\varepsilon\oplus Z_\varepsilon$ and  $A_\varepsilon=A_\varepsilon^+\oplus A_\varepsilon^-$ according with Proposition \ref{Proposition_linear_estimates}. Then for $\varepsilon$ sufficiently small (we still denote $\varepsilon\in (0,\varepsilon_0]$), there is an invariant manifold $\mathcal{M}_\varepsilon$ for \eqref{semilinear_problem}, which is given by graph of a certain Lipschitz continuous map $s_\ast^\varepsilon:Y_\varepsilon\to Z_\varepsilon$ as 
$$
\mathcal{M}_\varepsilon=\{u^\varepsilon\in X_\varepsilon\,;\, u^\varepsilon = Q_\varepsilon(\bar{\lambda})u^\varepsilon+s_{*}^\varepsilon(Q_\varepsilon(\bar{\lambda})u^\varepsilon)\}
.$$ 
The map $s_\ast^\varepsilon:Y_\varepsilon\to Z_\varepsilon$ satisfies the condition
\begin{equation}\label{estimate_invariant_manifold}
|\!|\!|s_\ast^\varepsilon |\!|\!|=\sup_{v^\varepsilon\in Y_\varepsilon}\|s_\ast^\varepsilon(v^\varepsilon)\|_{X_\varepsilon^\frac{1}{2}}\leq C(\tau(\varepsilon)+p(\varepsilon)^{-\frac{1}{2}}),
\end{equation}
for some constant $C$ independent of $\varepsilon$.
The invariant manifold $\mathcal{M}_\varepsilon$ is exponentially attracting and the global attractor $\mathcal{A}_\varepsilon$ of the problem \eqref{semilinear_problem} lying in $\mathcal{M}_\varepsilon$ and the flow on $\mathcal{A}_\varepsilon$ is given by
$$
u^\varepsilon(t)=v^\varepsilon(t)+s_\ast^\varepsilon(v^\varepsilon(t)), \quad t\in\R,
$$ 
where $v^\varepsilon(t)$ satisfy
$$
\dot{v^\varepsilon}+A_\varepsilon^+v^\varepsilon=Q_\varepsilon(\bar{\lambda})f(v^\varepsilon+s_\ast^\varepsilon(v^\varepsilon(t))).
$$
\end{theo}
\begin{proof}
The demonstration is standard so we omit some calculations. We focus on estimate \eqref{estimate_invariant_manifold}.

Given $L,\Delta>0$ we consider the set  
$$
\Sigma_\varepsilon = \Big\{s^\varepsilon: Y_\varepsilon\to Z_\varepsilon\,;\, |\!|\!| s^\varepsilon|\!|\!|\leq D  \tn{ and }\|s^\varepsilon(v)-s^\varepsilon(\tilde{v})\|_{X_\varepsilon^{\frac{1}{2}}}\leq \Delta \|v-\tilde{v}\|_{X_\varepsilon^{\frac{1}{2}}}\Big\}.
$$
$(\Sigma_\varepsilon, |\!|\!| \cdot |\!|\!|)$ is a complete metric space. We write the solution $u^\varepsilon$ of \eqref{semilinear_problem} as $u^\varepsilon=v^\varepsilon+z^\varepsilon$, with $v^\varepsilon\in Y_\varepsilon$ and $z^\varepsilon\in Z_\varepsilon$ and since $Q_\varepsilon(\bar{\lambda})$ and $I-Q_\varepsilon(\bar{\lambda})$ commute with $A_\varepsilon$,  we obtain the equations 
\begin{equation}\label{couple_system}
\begin{cases}
v_t^\varepsilon+A_\varepsilon^+ v^\varepsilon = Q_\varepsilon(\bar{\lambda})f(v^\varepsilon+z^\varepsilon):=H_\varepsilon(v^\varepsilon,z^\varepsilon)\\
z_t^\varepsilon+A_\varepsilon^- z^\varepsilon = (I-Q_\varepsilon(\bar{\lambda}))f(v^\varepsilon+z^\varepsilon):=G_\varepsilon(v^\varepsilon,z^\varepsilon).
\end{cases}
\end{equation}
By assumption there is a certain $\rho>0$ such that for all $ v^\varepsilon,\tilde{v}^\varepsilon\in Y_\varepsilon$ and $z^\varepsilon,\tilde{z}^\varepsilon\in Z_\varepsilon$,
\begin{itemize} 
\item[]$\|H_\varepsilon( v^\varepsilon,z^\varepsilon)\|_{X_\varepsilon^\frac{1}{2}}\leq \rho,$ $\|G_\varepsilon( v^\varepsilon,z^\varepsilon)\|_{X_\varepsilon^\frac{1}{2}}\leq \rho,$
\item[]$\|H_\varepsilon( v^\varepsilon,z^\varepsilon)-H_\varepsilon( \tilde{v}^\varepsilon,\tilde{z}^\varepsilon)\|_{X_\varepsilon^\frac{1}{2}}\leq \rho (\|v^\varepsilon-\tilde{v}_\varepsilon\|_{X_\varepsilon^\frac{1}{2}} +\|z^\varepsilon-\tilde{z}_\varepsilon \|_{X_\varepsilon^\frac{1}{2}} ),$
\item[]$
\|G_\varepsilon( v^\varepsilon,z^\varepsilon)-G_\varepsilon( \tilde{v}^\varepsilon,\tilde{z}^\varepsilon)\|_{X_\varepsilon^\frac{1}{2}}\leq \rho (\|v^\varepsilon-\tilde{v}_\varepsilon\|_{X_\varepsilon^\frac{1}{2}} +\|z^\varepsilon-\tilde{z}_\varepsilon \|_{X_\varepsilon^\frac{1}{2}} ).$
\end{itemize}
Also, for $\varepsilon$ sufficiently small, we can choose $\rho$ such that
 \begin{itemize}
\item[] $\rho M\beta^{-1}\leq D$, $0 \leq \beta-\gamma-\rho M(1+\Delta)$, $\frac{\rho M^2(1+\Delta)}{\beta-\gamma-\rho M(1+\Delta)}\leq \Delta$,
\item[] $\rho M \beta^{-1}+\frac{\rho^2M^2(1+\Delta)\gamma^{-1}}{\beta-\gamma-\rho M(1+\Delta)} \leq \frac{1}{2}$, $L=\Big[\rho M+\frac{\rho^2 M^2(1+\Delta)(1+M)}{\beta-\gamma-\rho M(1+\Delta)} \Big]$, $\beta-L> 0$.
\end{itemize}

Let $s^\varepsilon\in \Sigma_\varepsilon$ and $v^\varepsilon(t)=v^\varepsilon(t,\tau,\eta,s^\varepsilon)$ be the solution of 
$$
\begin{cases}
v_t^\varepsilon+A_\varepsilon^+ v^\varepsilon=H_\varepsilon(v^\varepsilon,s^\varepsilon(v^\varepsilon)),\quad  t<\tau \\
v^\varepsilon(\tau)=\eta.
 \end{cases}
$$
We define $\Phi_\varepsilon: \Sigma_\varepsilon\to\Sigma_\varepsilon$ given by
$$
\Phi_\varepsilon(s^\varepsilon)(\eta)=\int_{-\infty}^{\tau} e^{-A_\varepsilon^- (\tau - r)}G_\varepsilon(v^\varepsilon(r),s^\varepsilon(v^\varepsilon(r)) )\,dr.
$$
Then by Proposition \ref{Proposition_linear_estimates}, $\|\Phi_\varepsilon(s^\varepsilon)(\eta)\|_{X_\varepsilon^{\frac{1}{2}}}\leq D$.
 
For $s^\varepsilon, \tilde{s^\varepsilon}\in \Sigma_\varepsilon$, $\eta,\tilde{\eta}\in Y_\varepsilon$,  $v^\varepsilon(t)=v^\varepsilon(t,\tau,\eta,s^\varepsilon)$ and $\tilde{v}^\varepsilon(t)=\tilde{v}^\varepsilon(t,\tau,\tilde{\eta},\tilde{s}^\varepsilon)$ we have
\begin{multline*}
v^\varepsilon(t)-\tilde{v}^\varepsilon(t) = e^{-A_\varepsilon^+ (t-\tau)}(\eta-\tilde{\eta}) \\ +\int_{\tau}^t   e^{-A_\varepsilon^+ (t - r)}[H_\varepsilon(v^\varepsilon(r),s^\varepsilon(v^\varepsilon(r)) )-H_\varepsilon(\tilde{v}^\varepsilon(r),\tilde{s}^\varepsilon(\tilde{v}^\varepsilon(r)) )] \,dr,
\end{multline*}
and we can prove that
$$
\|v^\varepsilon(t)-\tilde{v}^\varepsilon(t)\|_{X_\varepsilon^{\frac{1}{2}}}\leq \Big[M\|\eta-\tilde{\eta}\|_{X_\varepsilon^\frac{1}{2}}+\rho M \gamma^{-1}|\!|\!| s^\varepsilon-\tilde{s}^\varepsilon |\!|\!|  \Big]  e^{[\rho M(1+\Delta)+\gamma](\tau-t)}.
$$
From this we obtain 
\begin{multline*}
\|\Phi_\varepsilon(s^\varepsilon)(\eta)-\Phi_\varepsilon(\tilde{s}^\varepsilon)(\tilde{\eta})\|_{X_\varepsilon^{\frac{1}{2}}}\\ \leq \Big[\frac{\rho M^2(1+\Delta)}{\beta-\gamma-\rho M(1+\Delta)}\Big] \|\eta-\tilde{\eta}\|_{X_\varepsilon^{\frac{1}{2}}}+\Big[\rho M\beta^{-1}+\frac{\rho^2M^2(1+\Delta)\gamma^{-1}}{\beta-\gamma-\rho M (1+\Delta)}\Big]|\!|\!|s^\varepsilon-\tilde{s}^\varepsilon |\!|\!|.
\end{multline*}
Therefore $\Phi_\varepsilon$ is a contraction on $\Sigma_\varepsilon$ hence there is a unique  $s_\ast^\varepsilon \in \Sigma_\varepsilon$. 

Now, let $(\bar{v}^\varepsilon,\bar{z}^\varepsilon)\in \mathcal{M}_\varepsilon$, $\bar{z}^\varepsilon=s_\ast^\varepsilon(\bar{v}^\varepsilon)$ and let $v_{s_\ast}^\varepsilon(t)$ be the solution of 
$$
\begin{cases}
v_t^\varepsilon+A_\varepsilon^+ v^\varepsilon=H_\varepsilon(v^\varepsilon,s_*^\varepsilon(v^\varepsilon)),\quad  t<\tau \\
v^\varepsilon(0)=\bar{v}^\varepsilon.
 \end{cases}
$$ 
Thus, $\{(v_{s_*}^\varepsilon(t), s_*^\varepsilon(v_{s_*}^\varepsilon(t))\}_{t\in\R}$ defines a curve on $\mathcal{M}_\varepsilon$. But the only solution of equation 
$$
z_t^\varepsilon+A_\varepsilon^- z^\varepsilon=G_\varepsilon(v_{s_*}^\varepsilon(t),s_*^\varepsilon(v_{s_*}^\varepsilon(t)))
$$ 
which stay bounded when $t\to-\infty$ is given by
$$
z_{s_*}^\varepsilon=\int_{-\infty}^t e^{-A_\varepsilon^- (t-r)}G_\varepsilon(v_{s_*}^\varepsilon(t),s_*^\varepsilon(v_{s_*}^\varepsilon(t)))\,dr = s_*^\varepsilon(v_{s_*}^\varepsilon(t)).
$$
Therefore $(v_{s_*}^\varepsilon(t), s_*^\varepsilon(v_{s_*}^\varepsilon(t))$ is a solution of \eqref{couple_system} through $(\bar{v}^\varepsilon,\bar{z}^\varepsilon)$ and thus $\mathcal{M}_\varepsilon$ is a invariant manifold for \eqref{semilinear_problem} according to the Definition \ref{invariant_manifold}.  

Now we will prove the estimate \eqref{estimate_invariant_manifold}. Note that since $Q_0(\lambda)=I_{X_0^\frac{1}{2}}$ we have $(I-Q_0(\lambda))=0$, $\mathcal{M}_0=X_0^\frac{1}{2}$, $s_\ast^0=0$ and $G_0(v^0,s_*^0(v^0))=0$, where $v^0$ is solution of \eqref{couple_system} with $\varepsilon=0$, that is $v^0=u^0$ is the solution of \eqref{semilinear_problem} with $\varepsilon=0$. Thus 
\begin{align*}
\|s^\varepsilon_*(\eta)\|_{X_\varepsilon^\frac{1}{2}}& \leq \int_{-\infty}^\tau \|e^{-A_\varepsilon^-(\tau-r)}G_\varepsilon(v^\varepsilon,s_*^\varepsilon(v^\varepsilon))\|_{X_\varepsilon^\frac{1}{2}}\,dr\\
& \leq \int_{-\infty}^\tau \|e^{-A_\varepsilon^-(\tau-r)}G_\varepsilon(v^\varepsilon,s_*^\varepsilon(v^\varepsilon))-e^{-A_\varepsilon^-(\tau-r)}G_\varepsilon(v^0,0)\|_{X_\varepsilon^\frac{1}{2}}\,dr\\
& + \int_{-\infty}^\tau \|e^{-A_\varepsilon^-(\tau-r)}G_\varepsilon(v^0,0)\|_{X_\varepsilon^\frac{1}{2}}\,dr.
\end{align*}
If we denote the last two integrals for $I_1$ and $I_2$ respectively, we have
\begin{align*}
I_1&\leq \int_{-\infty}^\tau M e^{-\beta (\tau - r)}\rho [(1+\Delta)\|v^\varepsilon-v^0\|_{X_\varepsilon^\frac{1}{2}}+|\!|\!|s_*^\varepsilon|\!|\!|]\,dr\\
&\leq \rho M(1+\Delta)\int_{-\infty}^\tau e^{-\beta (\tau - r)}\|v^\varepsilon-v^0\|_{X_\varepsilon^\frac{1}{2}}\,dr \\
& +\rho M |\!|\!|s_*^\varepsilon|\!|\!|\int_{-\infty}^\tau  e^{-\beta (\tau - r)}\,dr\\
&= \rho M(1+\Delta)\int_{-\infty}^\tau e^{-\beta (\tau-r)}\|v^\varepsilon-v^0\|_{X_\varepsilon^\frac{1}{2}}\,dr +\rho M\beta^{-1} |\!|\!|s_*^\varepsilon|\!|\!|.
\end{align*}
For $I_2$, we have $ G_\varepsilon(v^0,0)=(I-Q_\varepsilon(\bar{\lambda}))f(v^0)$, therefore
$
I_2\leq C (\tau(\varepsilon)+p(\varepsilon)^{-\frac{1}{2}}),
$
for some constant $C$ independent of $\varepsilon$. Thus
\begin{align*}
\|s^\varepsilon_*(\eta)\|_{X_\varepsilon^{\frac{1}{2}}} & \leq C(\tau(\varepsilon)+p(\varepsilon)^{-\frac{1}{2}})+\rho M\beta^{-1} |\!|\!|s_*^\varepsilon|\!|\!| + \rho M(1+\Delta)\int_{-\infty}^\tau e^{-\beta(\tau-r)}\|v^\varepsilon-v^0\|_{X_\varepsilon^{\frac{1}{2}} }\,dr.
\end{align*}
But, with the same argument above, it follows that 
$$
\|v^\varepsilon(t)-v^0(t)\|_{X_\varepsilon^{\frac{1}{2}}} \leq [C (\tau(\varepsilon)+p(\varepsilon)^{-\frac{1}{2}})+\rho M \gamma^{-1}|\!|\!|s_*^\varepsilon|\!|\!|]e^{[\rho M(1+\Delta)+\gamma](\tau-t)},
$$
thus
\begin{align*}
&\|s_*^\varepsilon(\eta)\|_{X_\varepsilon^\frac{1}{2}}\leq C(\tau(\varepsilon)+p(\varepsilon)^{-\frac{1}{2}})+ \Big[\rho M\beta^{-1}+\frac{\rho^2M^2(1+\Delta)\gamma^{-1}}{\beta-\gamma-\rho M (1+\Delta)} \Big]|\!|\!|s_*^\varepsilon|\!|\!|.
\end{align*}
Therefore
$
|\!|\!|s_*^\varepsilon|\!|\!|\leq C(\tau(\varepsilon)+p(\varepsilon)^{-\frac{1}{2}}).
$

It remains that $\mathcal{M}_\varepsilon$ is exponentially attracting and $\mathcal{A}_\varepsilon\subset\mathcal{M}_\varepsilon$. Let $(v^\varepsilon,z^\varepsilon)\in Y_\varepsilon\oplus Z_\varepsilon$ be the solution of \eqref{couple_system} and define $\xi^\varepsilon(t)=z^\varepsilon-s_*^\varepsilon(v^\varepsilon(t))$ and consider $y^\varepsilon(r,t), r\leq t$, $t\geq 0$, the solution of 
$$
\begin{cases}
y_t^\varepsilon+A_\varepsilon^+ y^\varepsilon=H_\varepsilon(y^\varepsilon,s_*^\varepsilon(y^\varepsilon)),\quad  r\leq t \\
y^\varepsilon(t,t)=v^\varepsilon(t).
\end{cases}
$$
Thus,
\begin{align*}
\|y^\varepsilon(r,t)-&v^\varepsilon(r)\|_{X_\varepsilon^{\frac{1}{2}}} \\
& =\Big\|\int_t^r e^{-A_\varepsilon^+(r-\theta)}[H_\varepsilon(y^\varepsilon(\theta,t),s_*^\varepsilon(y^\varepsilon(\theta,t)))-H_\varepsilon(v^\varepsilon(\theta),z^\varepsilon
(\theta))]\,d\theta\Big\|_{X_\varepsilon^{\frac{1}{2}}} \\
&\leq \rho M \int_r^t e^{-\gamma(r-\theta)}[(1+\Delta)\|y^\varepsilon(\theta,t)-v^\varepsilon(\theta)\|_{X_\varepsilon^{\frac{1}{2}}}+\|\xi^\varepsilon(\theta)\|_{X_\varepsilon^{\frac{1}{2}}}]\,d\theta
\end{align*}
By Gronwall inequality
$$
\|y^\varepsilon(r,t)-v^\varepsilon(r)\|_{X_\varepsilon^{\frac{1}{2}}}\leq\rho M \int_r^{t} e^{-(-\gamma-\rho M(1+\Delta))(\theta-r)} \|\xi^\varepsilon(\theta)\|_{X_\varepsilon^{\frac{1}{2}}}\,d\theta\quad r\leq t.
$$
Now we take $t_0\in[r,t]$ and then
\begin{align*}
\|y^\varepsilon(r,t) &-y^\varepsilon(r,t_0)\|_{X_\varepsilon^{\frac{1}{2}}}\\ 
& = \|e^{-A_\varepsilon^+ (r-t_0)}[y(t_0,t)-v^\varepsilon(t_0)]\|_{X_\varepsilon^{\frac{1}{2}}} \\
& +\Big\| \int_{t_0}^r e^{-A_\varepsilon^+(r-\theta)}[H_\varepsilon(y^\varepsilon(\theta,t),s_*^\varepsilon(y^\varepsilon(\theta,t)))-H_\varepsilon(y^\varepsilon(\theta,t_0),s_*^\varepsilon(y^\varepsilon(\theta,t_0)))]\, d\theta\Big\|_{X_\varepsilon^{\frac{1}{2}}} \\
&\leq \rho M^2 e^{-\gamma (r-t_0)} \int_{t_0}^{t} e^{-(-\gamma-\rho M(1+\Delta))(\theta-t_0)}\|\xi^\varepsilon(\theta)\|_{X_\varepsilon^{\frac{1}{2}}}\,d\theta\\
&+\rho M\int_{r}^{t_0} e^{-\gamma(r-\theta)}(1+\Delta)\| y^\varepsilon(\theta,t)-y^\varepsilon(\theta,t_0)\|_{X_\varepsilon^{\frac{1}{2}}}\,d\theta. 
\end{align*}
By Gronwall inequality
$$
\|y^\varepsilon(r,t)-y^\varepsilon(r,t_0)\|_{X_\varepsilon^{\frac{1}{2}}} \leq \rho M^2 \int_{t_0}^{t} e^{-(-\gamma-\rho M(1+\Delta))(\theta-r)} \|\xi^\varepsilon(\theta)\|_{X_\varepsilon^\frac{1}{2}}\,d\theta.
$$
Since
$$
z^\varepsilon(t)=e^{-A_\varepsilon^- (t-t_0)}z^\varepsilon(t_0)+\int_{t_0}^t e^{-A_\varepsilon^- (t-r)}G_\varepsilon(v^\varepsilon(r),z^\varepsilon(r))\,dr,
$$
we can estimate $\xi^\varepsilon(t)$ as
\begin{align*}
e^{\beta (t-t_0)} \|\xi^\varepsilon(t)\|_{X_\varepsilon^{\frac{1}{2}}} & \leq M \|\xi^\varepsilon(t_0)\|_{X_\varepsilon^{\frac{1}{2}}}+ \Big[\rho M+\frac{\rho^2 M^2(1+\Delta)}{\beta-\gamma-\rho M(1+\Delta) }\Big]\int_{t_0}^t e^{\beta(r-t_0)}\|\xi^\varepsilon(r)\|_{X_\varepsilon^{\frac{1}{2}}} \,dr\\
& +\frac{\rho^2M^3(1+\Delta)}{\beta-\gamma-\rho M (1+\Delta)}\int_{t_0}^t e^{-(\beta-\gamma-\rho M(1+\Delta)(\theta-t_0)}e^{\beta (\theta-t_0)}\|\xi^\varepsilon(\theta)\|_{X_\varepsilon^{\frac{1}{2}}}\,d\theta\\
&\leq M \|\xi^\varepsilon(t_0)\|_{X_\varepsilon^{\frac{1}{2}}}+\Big[\rho M+\frac{\rho^2 M^2(1+\Delta)(1+M)}{\beta-\gamma-\rho M(1+\Delta)}\Big]\int_{t_0}^t e^{\beta(r-t_0)}\|\xi^\varepsilon(r)\|_{X_\varepsilon^{\frac{1}{2}}}\,dr.
\end{align*}
By Gronwall inequality
$$
\|\xi^\varepsilon(t)\|_{X_\varepsilon^{\frac{1}{2}}}\leq M\|\xi^\varepsilon(t_0)\|_{X_\varepsilon^{\frac{1}{2}}}e^{-(L-\beta)(t-t_0)},
$$
and then
$$
\|z^\varepsilon(t)-s_*^\varepsilon(v^\varepsilon(t))\|_{X_\varepsilon^{\frac{1}{2}}}=\|\xi^\varepsilon(t)\|_{X_\varepsilon^{\frac{1}{2}}}\leq M\|\xi^\varepsilon(t_0)\|_{X_\varepsilon^{\frac{1}{2}}}e^{-(L-\beta)(t-t_0)}.
$$

Now if $u^\varepsilon:=T_\varepsilon(t)u_0^\varepsilon=v^\varepsilon(t)+z^\varepsilon(t)$, $t\in\R$, denotes the solution through at $u_0^\varepsilon=v_0^\varepsilon+z_0^\varepsilon\in \mathcal{A}_\varepsilon$, then
$$
\|z^\varepsilon(t)-s_*^\varepsilon(v^\varepsilon(t))\|_{X_\varepsilon^{\frac{1}{2}}}\leq M\|z_0^\varepsilon-s_*^\varepsilon(v_0^\varepsilon)\|_{X_\varepsilon^{\frac{1}{2}}}e^{-(L-\beta)(t-t_0)}.
$$
Since $\{T_\varepsilon(t)u_0^\varepsilon\,;\,t\in\R\}\subset\mathcal{A}_\varepsilon$ is bounded, letting $t_0\to-\infty$ we obtain $T_\varepsilon(t)u_0^\varepsilon=v^\varepsilon(t)+s_*^\varepsilon(v^\varepsilon(t))\in\mathcal{M}_\varepsilon$. That is $\mathcal{A}_\varepsilon\subset\mathcal{M}_\varepsilon$. Moreover, if $B_\varepsilon\subset X_\varepsilon^\frac{1}{2}$ is a bounded set and  $u_0^\varepsilon=v_0^\varepsilon+z_0^\varepsilon\in B_\varepsilon$, and we conclude that $T_\varepsilon(t)u_0^\varepsilon=v^\varepsilon(t)+z^\varepsilon(t)$ satisfies
\begin{align*}
\sup_{u_0^\varepsilon\in B_\varepsilon}\inf_{w\in\mathcal{M}_\varepsilon}\|T_\varepsilon(t)u_0^\varepsilon-w\|_{X_\varepsilon^\frac{1}{2}}&\leq \sup_{u_0^\varepsilon\in B_\varepsilon}\|z^\varepsilon(t)-s_*^\varepsilon(v^\varepsilon(t))\|_{X_\varepsilon^\frac{1}{2}}\\
&\leq Me^{-(L-\beta)(t-t_0)} \sup_{u_0^\varepsilon\in B_\varepsilon}\|z_0^\varepsilon-s_*^\varepsilon(v_0^\varepsilon)\|_{X_\varepsilon^{\frac{1}{2}}},
\end{align*}
which implies
$$
\tn{dist}_H(T_\varepsilon(t)B_\varepsilon,\mathcal{M}_\varepsilon)\leq C(B_\varepsilon)e^{-(L-\beta)(t-t_0)},
$$
and thus the proof is complete.
\end{proof}

\section{Rate of Convergence of Attractors}\label{Rate_of_Convergence_of_Attractors}
In this section we will estimate the continuity of attractors of \eqref{nonlinear_semigroup} in the Hausdorff metric by rate of convergence of resolvent operators obtained in the Section \ref{Rate of Convergence of Resolvent and Equilibria}.

\begin{theo}\label{rate_attractors} Let $\mathcal{A}_\varepsilon$, $\varepsilon\in[0,\varepsilon_0]$, be the attractor for \eqref{nonlinear_semigroup}. Then there is a positive constant $C$ independent of $\varepsilon$ such that
$$
\tn{d}_H(\mathcal{A}_\varepsilon,\mathcal{A}_0)\leq C (\tau(\varepsilon)+p(\varepsilon)^{-\frac{1}{2}}).
$$
\end{theo}
\begin{proof} By triangle inequality,
\begin{align*}
\tn{d}_H(\mathcal{A}_\varepsilon,\mathcal{A}_0)&\leq \tn{d}_H(\mathcal{A}_\varepsilon,Q_\varepsilon(\bar{\lambda})\mathcal{A}_\varepsilon)+
\tn{d}_H(Q_\varepsilon(\bar{\lambda})\mathcal{A}_\varepsilon,\mathcal{A}_0)\\
& \leq \tn{d}_H(\mathcal{A}_\varepsilon,Q_\varepsilon(\bar{\lambda})\mathcal{A}_\varepsilon)+
\tn{d}_H(Q_\varepsilon(\bar{\lambda})\mathcal{A}_\varepsilon,PQ_\varepsilon(\bar{\lambda})\mathcal{A}_\varepsilon)+
\tn{d}_H(PQ_\varepsilon(\bar{\lambda})\mathcal{A}_\varepsilon,\mathcal{A}_0).
\end{align*}

We will estimate each part. 

Let $z\in\mathcal{A}_\varepsilon$, since $\mathcal{A}_\varepsilon\subset\mathcal{M}_\varepsilon$, $z=z^\varepsilon+s_*^\varepsilon(z^\varepsilon)$ for some $z^\varepsilon\in Q_\varepsilon(\bar{\lambda})\mathcal{A}_\varepsilon$, then
\begin{align*}
\tn{dist}(z,Q_\varepsilon(\bar{\lambda})\mathcal{A}_\varepsilon)&=\inf_{x\in Q_\varepsilon(\bar{\lambda})\mathcal{A}_\varepsilon}\|z-x\|_{X_\varepsilon^\frac{1}{2}}\leq \|z^\varepsilon+s_*^\varepsilon(z^\varepsilon)-z^\varepsilon\|_{X_\varepsilon^\frac{1}{2}}=\|s_*^\varepsilon(z^\varepsilon)\|_{X_\varepsilon^\frac{1}{2}}\\
&\leq |\!|\!|s_*^\varepsilon|\!|\!|\leq C(\tau(\varepsilon)+p(\varepsilon)^{-\frac{1}{2}}),
\end{align*} 
\begin{align*}
\tn{dist}(z^\varepsilon,\mathcal{A}_\varepsilon)&=\inf_{x\in \mathcal{A}_\varepsilon}\|z^\varepsilon
-x\|_{X_\varepsilon^\frac{1}{2}}\leq \|z^\varepsilon-(z^\varepsilon+s_*^\varepsilon(z^\varepsilon))\|_{X_\varepsilon^\frac{1}{2}}=\|s_*^\varepsilon(z^\varepsilon)\|_{X_\varepsilon^\frac{1}{2}}\\
&\leq |\!|\!|s_*^\varepsilon|\!|\!|\leq C(\tau(\varepsilon)+p(\varepsilon)^{-\frac{1}{2}}),
\end{align*} 
which implies 
$$
\tn{d}_H(\mathcal{A}_\varepsilon,Q_\varepsilon(\bar{\lambda})\mathcal{A}_\varepsilon)= \tn{dist}_H(\mathcal{A}_\varepsilon,Q_\varepsilon(\bar{\lambda})\mathcal{A}_\varepsilon)+
\tn{dist}_H(Q_\varepsilon(\bar{\lambda})\mathcal{A}_\varepsilon,\mathcal{A}_\varepsilon)\leq C(\tau(\varepsilon)+p(\varepsilon)^{-\frac{1}{2}}).
$$
 
Let $z^\varepsilon\in Q_\varepsilon(\bar{\lambda})\mathcal{A}_\varepsilon$ then $z^\varepsilon=Q_\varepsilon(\bar{\lambda})w^\varepsilon$ for some $w^\varepsilon\in\mathcal{A}_\varepsilon$. Since $\bigcup_{\varepsilon\in[0,\varepsilon_0]}\mathcal{A}_\varepsilon$ is uniformly bounded in $X_\varepsilon^\frac{1}{2}$ we can assume $\|w^\varepsilon\|_{X_\varepsilon^\frac{1}{2}}\leq C$. Thus
\begin{align*}
\tn{dist}(z^\varepsilon,PQ_\varepsilon(\bar{\lambda})\mathcal{A}_\varepsilon)&=\inf_{x\in PQ_\varepsilon(\bar{\lambda})\mathcal{A}_\varepsilon}\|z^\varepsilon
-x\|_{X_\varepsilon^\frac{1}{2}}\leq \|z^\varepsilon-Pz^\varepsilon\|_{X_\varepsilon^\frac{1}{2}}=\|Q_\varepsilon(\bar{\lambda})w^\varepsilon-PQ_\varepsilon(\bar{\lambda})w^\varepsilon\|_{X_\varepsilon^\frac{1}{2}}\\
&=\|(Q_\varepsilon(\bar{\lambda})-P)Q_\varepsilon(\bar{\lambda})w^\varepsilon\|_{X_\varepsilon^\frac{1}{2}}\leq C(\tau(\varepsilon)+p(\varepsilon)^{-\frac{1}{2}}),
\end{align*} 
\begin{align*}
\tn{dist}(Pz^\varepsilon,Q_\varepsilon(\bar{\lambda})\mathcal{A}_\varepsilon)&=\inf_{x\in Q_\varepsilon(\bar{\lambda})\mathcal{A}_\varepsilon}\|Pz^\varepsilon
-x\|_{X_\varepsilon^\frac{1}{2}}\leq \|Pz^\varepsilon-z^\varepsilon\|_{X_\varepsilon^\frac{1}{2}}=\|PQ_\varepsilon(\bar{\lambda})w^\varepsilon-Q_\varepsilon(\bar{\lambda})w^\varepsilon\|_{X_\varepsilon^\frac{1}{2}}\\
&=\|(P-Q_\varepsilon(\bar{\lambda}))Q_\varepsilon(\bar{\lambda})w^\varepsilon\|_{X_\varepsilon^\frac{1}{2}}\leq C(\tau(\varepsilon)+p(\varepsilon)^{-\frac{1}{2}}),
\end{align*} 
which implies 
\begin{align*}
\tn{d}_H(Q_\varepsilon(\bar{\lambda})\mathcal{A}_\varepsilon,PQ_\varepsilon(\bar{\lambda})\mathcal{A}_\varepsilon)&= \tn{dist}_H(Q_\varepsilon(\bar{\lambda})\mathcal{A}_\varepsilon,PQ_\varepsilon(\bar{\lambda})\mathcal{A}_\varepsilon)+
\tn{dist}_H(PQ_\varepsilon(\bar{\lambda})\mathcal{A}_\varepsilon,Q_\varepsilon(\bar{\lambda})\mathcal{A}_\varepsilon)\\
&\leq C(\tau(\varepsilon)+p(\varepsilon)^{-\frac{1}{2}}).
\end{align*}

Finally we have $\mathcal{A}_0=[x_*^{1,0},x_*^{p,0}]$ and $T_0(\cdot)$ is a Morse-Smale semigroup. Since Morse-Smale semigroup are stable (see \cite{Bortolan}) we can assume $\mathcal{A}_\varepsilon=[x_*^{1,\varepsilon},x_*^{m,\varepsilon}]$ that is $Q_\varepsilon(\bar{\lambda})\mathcal{A}_\varepsilon=[Q_\varepsilon(\bar{\lambda})x_*^{1,\varepsilon},Q_\varepsilon(\bar{\lambda})x_*^{m,\varepsilon}]$ which implies $PQ_\varepsilon(\bar{\lambda})\mathcal{A}_\varepsilon=[PQ_\varepsilon(\bar{\lambda})x_*^{1,\varepsilon},PQ_\varepsilon(\bar{\lambda})x_*^{m,\varepsilon}]$. Without loss of generality we can assume $PQ_\varepsilon(\bar{\lambda})x_*^{1,\varepsilon}\leq x_*^{1,0}$ and $PQ_\varepsilon(\bar{\lambda})x_*^{m,\varepsilon}\leq x_*^{m,0}$. Thus
$$
\tn{d}_H(PQ_\varepsilon(\bar{\lambda})\mathcal{A}_\varepsilon,\mathcal{A}_0)=\|PQ_\varepsilon(\bar{\lambda})x_*^{1,\varepsilon}- x_*^{1,0}\|_{X_\varepsilon^\frac{1}{2}}+\|PQ_\varepsilon(\bar{\lambda})x_*^{m,\varepsilon}- x_*^{m,0}\|_{X_\varepsilon^\frac{1}{2}}.
$$
But
\begin{align*}
\|PQ_\varepsilon(\bar{\lambda})x_*^{1,\varepsilon}- x_*^{1,0}\|_{X_\varepsilon^\frac{1}{2}}&=\|PQ_\varepsilon(\bar{\lambda})x_*^{1,\varepsilon}- Px_*^{1,0}\|_{X_\varepsilon^\frac{1}{2}}\leq \|P\|_{\LL(L^2)}\|Q_\varepsilon(\bar{\lambda})x_*^{1,\varepsilon}- Px_*^{1,0}\|_{X_\varepsilon^\frac{1}{2}}\\
&\leq \|P\|_{\LL(L^2)}\|Q_\varepsilon(\bar{\lambda})x_*^{1,\varepsilon}- Px_*^{1,\varepsilon}\|_{X_\varepsilon^\frac{1}{2}}+\|P\|_{\LL(L^2)}\|Px_*^{1,\varepsilon}-Px_*^{1,0}\|_{X_\varepsilon^\frac{1}{2}}\\
&\leq C(\tau(\varepsilon)+p(\varepsilon)^{-\frac{1}{2}}).
\end{align*}
In the same way
$
\|PQ_\varepsilon(\bar{\lambda})x_*^{m,\varepsilon}- x_*^{m,0}\|_{X_\varepsilon^\frac{1}{2}}\leq C(\tau(\varepsilon)+p(\varepsilon)^{-\frac{1}{2}}).
$
\end{proof}

\bibliographystyle{abbrv}
\bibliography{../Jabref/References}


\end{document}